\documentclass[11pt, reqno]{amsart}

\usepackage[latin1]{inputenc}
\usepackage{amsfonts}
\usepackage{amsmath}
\usepackage{amssymb}
\usepackage{amsthm}
\usepackage{fontenc}
\usepackage{graphicx}
\usepackage[all]{xy}
\usepackage[font=small, labelfont=bf]{caption}
\usepackage{subcaption}
\usepackage[colorlinks=true,urlcolor=blue,linkcolor=red,citecolor=magenta]{hyperref}
\usepackage{url}
\usepackage[left=1in,right=1in,top=1in,bottom=1in]{geometry}
\newtheorem{theorem}{Theorem}[section]
\newtheorem{lemma}[theorem]{Lemma}
\newtheorem{proposition}[theorem]{Proposition}

\newtheorem{corollary}[theorem]{Corollary}

\newtheorem{conjecture}{Conjecture}
\newtheorem{question}{Question}

\numberwithin{equation}{section}

\DeclareMathOperator{\Conv}{Conv} 
\DeclareMathOperator{\Aff}{Aff}

\DeclareMathOperator{\Iso}{Iso}

\DeclareMathOperator{\Stab}{Stab}
\DeclareMathOperator{\Gens}{Gens}

\DeclareMathOperator{\LinSym}{LinSym}
\DeclareMathOperator{\OSym}{OSym}

\begin{document}

\title{Inscribed Tverberg-Type Partitions for Orbit Polytopes}

\author{Steven Simon}
\address[SS]{Bard College, Department of Mathematics}

\email{ssimon@bard.edu}
\author{Tobias Timofeyev}
\address[TT]{University of Vermont, Deparment of Mathematics and Statistics}
\email{tobias.timofeyev@vt.edu}

\maketitle

\begin{abstract}  Tverberg's theorem states that any set of $t(r,d)=(r-1)(d+1)+1$ points in $\mathbb{R}^d$  can be partitioned into $r$ subsets whose convex hulls have non-empty $r$-fold intersection. Moreover, generic collections of fewer points cannot be so divided. Extending earlier work of the first author, we show that one can nonetheless guarantee inscribed ``polytopal partitions" with specified symmetry conditions in many such circumstances. Namely, for any faithful and full--dimensional orthogonal representation $\rho\colon G\rightarrow O(d)$ of any order $r$ group $G$, we show that a generic set of $t(r,d)-d$ points in $\mathbb{R}^d$ can be partitioned into $r$ subsets so that there are $r$ points, one from each of the resulting convex hulls, which are the vertices of a convex $d$--polytope whose isometry group contains $G$ via the regular action afforded by the representation. As with Tverberg's theorem, the number of points is optimal for this. At one extreme, this gives polytopal partitions for all regular $r$--gons in the plane, as well as for three of the six regular 4--polytopes in $\mathbb{R}^4$. At the other extreme, one has polytopal partitions for $d$-polytopes on $r$ vertices with isometry group equal to $G$ whenever $G$ is the isometry group of a vertex--transitive $d$-polytope.
\end{abstract}

\section{Introduction and Statement of Main Results}
\label{sec:1}

Tverberg's theorem~\cite{T66} establishes that any set of $t(r,d)=(r-1)(d+1)+1$ points in $\mathbb{R}^d$ can be partitioned into $r$ subsets whose convex hulls have non-empty $r$--fold intersection, called a Tverberg $r$--partition. The result is of fundamental importance in discrete and convex geometry, with a panoply of rich generalizations and variants in geometric combinatorics and beyond. We refer the reader to the surveys ~\cite{BSo19, BZ17,DELGMM19} for the history of Tverberg's theorem and its relatives, as well as a sampling of its many applications.

The Tverberg number $t(r,d)$ is very tight in that generic collections of fewer points (e.g., those in ``strong general position''~\cite{PS14}) do not admit a Tverberg $r$--partition. In such circumstances, it was asked in~\cite{LS21} whether one can nonetheless ensure ``inscribed'' polytopal variants of Tverberg's theorem. Namely, suppose that $P(r,d)$ is a prescribed  $d$--dimensional polytope on $r$ vertices in $\mathbb{R}^d$.  We say that a set of $N$ points in $\mathbb{R}^d$ can be \textit{$P(r,d)$-partitioned} (or simply \textit{polytopally partitioned} when the context is clear) if there exists a partition $A_1,\ldots, A_r$  of the set so that there are $r$ points $x_1\in \Conv(A_1),\ldots,  x_r\in \Conv(A_r)$, one from each of the resulting convex hulls, so that $\{x_1,\ldots, x_r\}$ is the vertex set of $P(r,d)$. As usual, we say that two polytopes in $\mathbb{R}^d$ are similar provided they are equal after the application of an isometry of $\mathbb{R}^d$ composed with uniform scaling. Lastly, we shall say that a property holds for a ``generic'' set of $N$ points in $\mathbb{R}^d$ provided the set of all $N$-tuples $(x_1,\ldots, x_N)$ of points in $\mathbb{R}^d$ which satisfy that property is a dense open set in $(\mathbb{R}^d)^N$ with full Lebesgue measure.

\begin{question}
\label{quest:quest} 
What is the minimum number $N_{P(r,d)}$ for which a generic set of $N$ points in $\mathbb{R}^d$ can be polytopally partitioned by a polytope similar to $P(r,d)$? \end{question} 

 Note that a generiticity condition on point sets in Question~\ref{quest:quest} is necessary, since for example any finite point set in $\mathbb{R}^d$ which lies in a hyperplane cannot be $P(r,d)$--partitioned by \textit{any} $d$-dimensional polytope on any number of vertices. Additionally, as it is natural to ask Question~\ref{quest:quest} with respect to other notions of equivalence for polytopes (for instance, affine or combinatorial equivalence), let us emphasize that the number $N_{P(r,d)}$ will \textit{always} be taken with respect to similarity. 

The main result of~\cite{LS21} gave tight results to Question~\ref{quest:quest} for regular polygons. Figure~\ref{fig:pic} above gives an example when $r=4$. 
	 
\begin{theorem}
\label{thm:polygon}
Let $r\geq 3$ and let $P_r=P(r,2)$ be a regular polygon. Then $N_{P_r}=t(r,2)-2=3r-4$. 
\end{theorem}

 \begin{figure}
  \includegraphics[scale=.7]{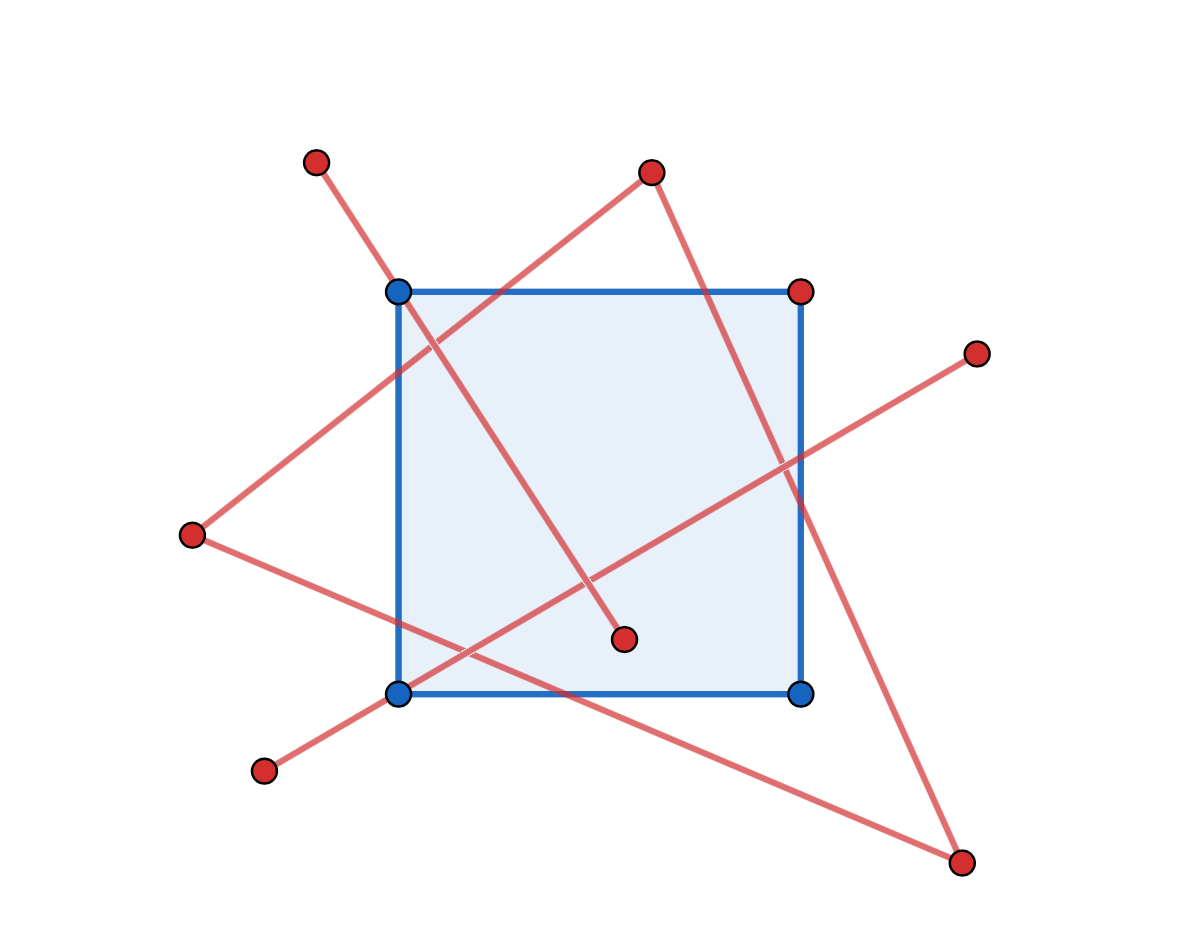}
  \caption{A $P_4$--partition of 8 (red) points in the plane.}
  \label{fig:pic}
\end{figure}

\subsection{Prescribed symmetries and orbit polytopes}

The aim of the present paper is to extend Theorem~\ref{thm:polygon} to $d$-polytopes with prescribed symmetry conditions. Observing that the cyclic group $\mathbb{Z}_r$ acts regularly (i.e., freely and transitively) on the vertex set of a regular $r$-gon, our central result Theorem~\ref{thm:orbit} affords polytopal partitions by $d$-polytopes whose isometry group contains an arbitrarily prescribed subgroup, and in fact one which acts regularly on the vertex set via a prescribed $d$-dimensional orthogonal representation. 

Let $\rho\colon G\rightarrow O(d)$ be a faithful orthogonal representation of a finite group. Let $u\in \mathbb{R}^d\setminus\{0\}$ and let $a\in \mathbb{R}^d$. By an \textit{orbit polytope}
$$P(G;u,a)=\Conv(a+G\cdot u)$$ for $\rho$, we shall mean the convex hull of the translate of a $G$-orbit under the orthogonal action afforded by the representation. Such polytopes are very well--studied (see, e.g.~\cite{Ba77,CrKe06,FL16,FL18,On93,SSS11}), particularly so for finite reflection groups~\cite{BaVe88,BoBo10,Ho12,HoLaTh11}, and have origins dating back to the classical Wythoff construction for uniform polytopes and the study of Coxeter groups~\cite{Cox34}. As  in~\cite[Section 2]{FL16}, the vertex set of $P(G;u,a)$ is by definition contained in $a+G\cdot u$ while the action of any $g\in G$ determines an isometry of $P(G;u,a)$, so the vertex set of $P(G;u,a)$ is precisely $a+G\cdot u$. In particular, any orbit polytope is vertex--transitive, and $G\le \Iso (P(G; u, a))$ via the action of the representation. 

For our purposes, we shall consider those representations which are \textit{full--dimensional}, that is, those for which there exists some orbit $G\cdot u$ whose affine hull is all of $\mathbb{R}^d$. Letting $\mathbb{R}[G]=\{\sum_{g\in G} r_g g\mid r_g\in \mathbb{R}\}$ denote the (left) regular representation and taking the module-theoretic perspective, it can be shown that a representation is full--dimensional precisely when it is isomorphic to a subrepresentation of $\mathbb{R}[G]$  which does not contain a trivial  subrepresentation~\cite[Theorem 4.3]{FL16}. Finally, we shall say that $P(G;u,a)$ is a \textit{free orbit polytope} if $u$ has trivial stabilizer, in which case the vertex set is identifiable with the group $G$ under the action of the representation. For example, if $\rho\colon \mathbb{Z}_r\rightarrow O(2)$ is the standard representation of the cyclic group $\mathbb{Z}_r$ given by rotations, then $\rho$ is full dimensional, every orbit polytope is free, and the resulting orbit polytopes are regular $r$-gons in the plane.

 \begin{theorem}
 \label{thm:orbit} Let $\rho\colon G\rightarrow O(d)$ be a faithful and full--dimensional representation of a finite group $G$ of order $r\geq 3$. Then a generic set of $t(r,d)-d=(r-2)(d+1)+2$ points in $\mathbb{R}^d$ can be $P(r,d)$-partitioned by a free orbit polytope for $\rho$. Moreover, such a partition fails for generic collections of fewer points in $\mathbb{R}^d$. \end{theorem} 
 
Note that Theorem~\ref{thm:orbit} applies to any finite group $G$, as one may consider the faithful and full-dimensional representation $\mathbb{R}^\perp[G]=\{\sum_{g\in G} r_g g \mid \sum_{g\in G} r_g=0\}$, the orthogonal complement of the trivial (diagonal) representation inside the regular representation $\mathbb{R}[G]$. In this case the free orbit polytopes are vertex and facet transitive $(|G|-1)$--dimensional simplicies with regular $G$ action on the vertex set. 
 
 \subsection{Summary of Examples} We provide some special cases of Theorem~\ref{thm:orbit}, which in particular gives the upper bounds  $N_{P(r,d)}\leq (r-2)(d+1)+2$ for a number of classical vertex-transitive $d$-polytopes. A detailed discussion of these examples is given in Section~\ref{sec:4}. 
 
 First, Theorem~\ref{thm:polygon} is recovered by considering the standard representation of finite cyclic groups as rotations in the plane. These groups can be identified with the unique finite subgroups of the unit complex numbers, and a similar analysis for the non-cyclic finite subgroups of the unit quaternions will be shown to yield polytopal partitions for three of the six regular 4-polytopes. As discussed in Section \ref{sec:4.1}, we expect each of the following estimates to exceed the exact value by 3. 
 \begin{theorem}
\label{thm:4poly}

Let $P_{\{3,3,4\}}$, $P_{\{3,4,3\}}$, and $P_{\{3,3,5\}}$ denote the regular 4-dimensional cross--polytope, regular 24--cell, and regular 600--cell, respectively. Then\\ 
(a) $N_{\{3,3,4\}}\leq 32$,\\
(b) $N_{\{3,4,3\}}\leq 112$, and\\
(c) $N_{\{3,3,5\}} \leq 592$.
\end{theorem}

In the examples above, the isometry group of the orbit polytope is larger than the group prescribed. Whether or not this is true in general depends on the group $G$ as well as the given representation. If $G$ is either (i) a finite abelian group with exponent greater than 2 or (ii) generalized dicyclic (see, e.g., ~\cite{Ba77,MSV15}), then $\Iso(P(G;u,a))>G$ for any free orbit polytope determined by any representation. On the other hand, a theorem of Babai~\cite{Ba77} shows that a group is realizable as the isometry group of a vertex--transitive $d$-polytope precisely when it is not of the form (i) or (ii). Using results of~\cite{Fr18, FL16, FL18},  we can ensure that $\Iso(P(G;u,a))=G$ in Theorem~\ref{thm:orbit} in these later cases.

\begin{theorem}
\label{thm:closed} 
Let $d\geq 2$, let $G$ be the isometry group of a vertex transitive $d$-polytope, and let $r=|G|$. Then a generic set of $(r-2)(d+1)+2$ points in $\mathbb{R}^d$ can be $P(r,d)$-partitioned by a vertex transitive $d$-polytope with $\Iso(P(r,d))=G$. 
\end{theorem}

 Notable examples of Theorem~\ref{thm:closed} include polytopal partitions for all Coxeter permutahedra, that is, the free orbit polytopes determined by the finite reflection groups (see Corollary~\ref{cor:Coxeter}). Considering dimension 3 along with other finite subgroups of the orthogonal group $O(3)$, we obtain polyhedral partitions for all right regular prisms and antiprisms, as well as for certain almost uniform copies of some Platonic and Archimedean solids. A complete discussion of the similarity--type of these polytopes is given in Section~\ref{sec:4.3}. 
 
 \begin{theorem}
 \label{thm:3poly} Let $D_n$ denote the Dihedral group and let $S_n$ and $A_n$ denote the  symmetric and alternating groups on $n$ elements, respectively.\\
(a) $N_{Pr(n)}\leq 8n-6$ for all $n\geq 3$, where $Pr(n)$ is a right regular prism with regular $n$-gon base; $\Iso(Pr(n))\cong D_n\times \mathbb{Z}_2$.\\
 (b) $N_{APr(n)}\leq 8n-6$ for all $n\geq 2$, where $APr(n)$ is a right regular antiprism with regular $n$-gon base; $\Iso(APr(n))\cong D_{2n}$.\\
 (c) $N_{tO}\leq 90$, where $tO$ is a truncated octahedron; $\Iso(tO)\cong S_4$.\\
(d) $N_{tCO}\leq 186 $, where $tCO$ is a truncated cuboctahedron; $\Iso(tCO)\cong S_4\times \mathbb{Z}_2$.\\
(e) $N_{tID}\leq 474$, where $tID$ is a truncated icosidodecahedron; $\Iso(tID)\cong A_5\times \mathbb{Z}_2$.\\
(f) $N_I\leq 42$, where $I$ is an icosahedron; $\Iso(I)\cong A_4$.\\
(g) $N_{sC}\leq 90$, where $sC$ is a snub cube; $\Iso(sC)\cong S_4$.\\
(h) $N_{sD}\leq 234$, where $sD$ is a snub dodecahedron; $\Iso(sD)\cong A_5$. 
\end{theorem}

We note that $APr(2)$ is a tetrahedron whose faces are congruent isosceles triangles. For (d)--(e) and (g)--(h), the isometry groups match those given by their \textit{uniform} versions~\cite[Chapter 11.5] {J18}. As discussed in Section~\ref{sec:4.3}, we expect the estimates of (a)--(b) and those of (c)--(h) to exceed the true values by 2 and 3, respectively.

\section{Orbit Partitions}
\label{sec:2}

Theorem~\ref{thm:orbit} is a special case of Theorem~\ref{thm:affine} below, which guarantees ``orbit partitions" for all representations that do not contain a trivial subrepresentation. The existence of such partitions and their optimality is given as Theorem~\ref{thm:affine} below, which by Lemmas \ref{lem:generic} and ~\ref{lem:iso} of the next section are shown to satisfy certain generiticity conditions. Theorem~\ref{thm:orbit} follows by considering faithful and full-dimensional representations. An example of Theorem~\ref{thm:affine} for non-faithful representations is given by Theorem~\ref{thm:Tobias}. 

As is often in the case in Tverberg-type theory, it will be convenient to recast the desired partition-type for sets of $N$ points in $\mathbb{R}^d$ as an equivalent statement in terms of affine maps $f\colon\Delta_{N-1}\rightarrow \mathbb{R}^d$ from an $(N-1)$--simplex to $\mathbb{R}^d$, that is, those which send any convex sum of vertices to the convex sums of their images. With this viewpoint, any partition of a set of $N$ points in $\mathbb{R}^d$ into $r$ subsets corresponds to $r$ pairwise faces of the simplex, while the convex hulls arising from the partition corresponds to the images of the $r$ faces under the mapping. For instance, a non-empty $r$-fold intersection of convex hulls as in Tverberg's theorem corresponds to a non-empty $r$-fold intersection of the images of $r$ pairwise disjoint faces. As an affine map is completely determined by the images of the vertices of the simplex, we shall say that a property holds for ``generic'' affine maps $f\colon \Delta_{N-1}\rightarrow \mathbb{R}^d$ provided the property holds for generic sets of $N$ points in $\mathbb{R}^d$. 

Recall that an $r$--tuple $(x_1,\ldots, x_r)$ of points from $\Delta_{N-1}$ is said to have pairwise disjoint support if there exist $r$ vertex-disjoint faces $(\sigma_1,\ldots, \sigma_r)$ of $\Delta_{N-1}$ with $x_j\in \sigma_j$ for all $1\leq j \leq r$. If $G$ is a group of order $r$, we parametrize each $r$--tuple $(x_g)_{g\in G}\subset \Delta_{N-1}$ of points with pairwise disjoint support by the group $G$. 

\begin{theorem}
\label{thm:affine}
Let $G$ be a finite group of order $r\geq 3$, let $\rho\colon G\rightarrow O(d)$ be an orthogonal representation which does not contain a trivial subrepresentation, and let $N=t(r,d)-d=(r-2)(d+1)+2$. Then for any affine map $f\colon\Delta_{N-1}\rightarrow \mathbb{R}^d$, there exists a collection $(x_g)_{g\in G}\subset \Delta_{N-1}$ of $r$ points with pairwise disjoint support and points $a$ and $u$ in $\mathbb{R}^d$ such that  \begin{equation}\label{eqn:orbit} f(x_g)=a+g\cdot u \end{equation} for all $g\in G$. Moreover, such a partition fails for generic affine maps if $N<(r-2)(d+1)+2$. 
\end{theorem}

Our proof of Theorem~\ref{thm:affine} is given over the course of Sections 2.1---2.4. The proof for the existence of such a partition  follows the usual Configuration--Space/Test--Map Paradigm of topological combinatorics. The central idea is to construct a naturally associated $G$--equivariant map $L\colon G^{\ast N}\rightarrow U$ whose zeros corresponding precisely to collections $(x_g)_{g\in G}\subset \Delta_{N-1}$ of $r$ points with pairwise disjoint support that satisfy equation~\ref{eqn:orbit}. The domain of this mapping, discussed in Section 2.1, is the deleted $r$--fold join $(\Delta_{N-1})^{\ast r}_\Delta$ commonly used in Tverberg-type theory, which as in~\cite{S2000} can be identified with the $N$-fold join $G^{\ast N}$ of the group itself.  The codomain $U=W\oplus \mathbb{R}^\perp[G]$ is the direct sum of a certain sub-representation $W$ of the $d$-fold product of the regular representation $\mathbb{R}[G]$ discussed in Section 2.2 with $\mathbb{R}^\perp[G]$, the orthogonal complement of the trivial representation inside $\mathbb{R}[G]$. The group $G$ acts on both the domain and codomain, and the map $L$ between these two spaces described in Section 2.3 respects these two actions. The action of $G$ on the domain is free, the representation space $U$ does not contain a trivial subrepresentation, and the map $L$ is affine, so $L$ is guaranteed to have a zero by Sarkaria's ``Linear Borsuk--Ulam" Theorem~\cite[Theorem 2.4]{S2000}, or equivalently from B\'ar\'any's colorful Carath\'edory Theorem ~\cite{B82}. This establishes the existence of an orbit partition, while the optimality of Theorem~\ref{thm:affine} is proved in Section 2.4 via a standard degrees of freedom count. 

\subsection{Configuration space} Let $N\geq 1$. The $r$-fold join $\Delta_{N-1}^{\ast r}$ of the simplex $\Delta_{N-1}$ consists of the join $\sigma_1\ast \cdots \ast \sigma_r$ of any $r$ faces $\sigma_1,\ldots, \sigma_r$ of $\Delta_{N-1}$ (including the possibility of empty faces), or in other words all formal sums of the form $\lambda_1x_1+\cdots + \lambda_rx_r$ where $x_j\in \sigma_j$ for all $1\leq j \leq r$, $\sum_{j=1}^r \lambda_j=1$, and $\lambda_j\geq 0$ for all $1\leq j \leq r$. The deleted $r$-fold join $$(\Delta_{N-1})^{\ast r}_\Delta = \{\lambda_1x_1+\cdots + \lambda_rx_r\in \sigma_1\ast \cdots \ast \sigma_r \mid x_i\in \sigma_i\,\,\,\, \text{and}\,\,\, \sigma_i\cap \sigma_j=\emptyset\,\, \text{for all} \,\, i\neq j\}$$ is the subcomplex of $\Delta_{N-1}^{\ast r}$ consisting of all points from the joins of pairwise disjoint faces. Given a group $G$ of order $r$, we follow Sarkaria~\cite{S2000} and parametrize each $r$-tuple of disjoint faces by $G$ and denote the resulting complex by $(\Delta_{N-1})^{\ast G}_\Delta$. For convenience, we shall make a slight abuse of notation and denote each formal sum $\sum_{g\in G}\lambda_gx_g\in (\Delta_{N-1})^{\ast G}_\Delta$ by $\lambda x$.

We shall consider the \textit{right} group action on $(\Delta_{N-1})^{\ast G}_\Delta$ induced by right multiplication in $G$. Thus $\lambda x\cdot h=\sum_g\lambda_{gh}x_{gh}$ for each $\lambda x=\sum_g \lambda_g x_g\in (\Delta_{N-1})^{\ast G}_\Delta$. This action is free because the faces are pairwise disjoint. Parametrizing $r$ disjoint copies of the vertex set $\{v_j\}_{j=1}^N$ of $\Delta_{N-1}$ by $G$, this simplicial complex can be  identified with the $N$-fold join $$G^{\ast N}=\left\{\sum_{j=1}^Nt_j v_j^{g_j} \mid g_j \in G,\,\, \sum_{j=1}^N t_j=1,\,\, \text{and}\,\, t_j\geq0\,\,\text{for all} \,\, 1\leq j\leq N \right\}.$$ For this complex, we consider the \textit{left} $G$-action given by $h\cdot \sum_{j=1}^N t_j v_j^{g_j} =\sum_{j=1}^N t_j v_j^{g_jh^{-1}}$ arising from right multiplication in $G$ by inverses. Again, this action is free. An equivariant isomorphism $\iota\colon G^{\ast\,N}\cong(\Delta_{N-1})^{\ast G}_\Delta$ between the complexes is given by grouping, which, as in~\cite{LS21}, has the following explicit description: For each $v=\sum_{j=1}^Nt_j v_j^{g_j}\in G^{\ast N}$ and each $g\in G$, let $$J_g(v)=\{j\mid g_j=g \,\, \text{and}\,\, t_j> 0\}.$$ One then defines $\iota(v)=\sum_{g\in G} \lambda_g x_g$, where $$\lambda_g=\sum_{j\in J_g(v)} t_j\,\,\, \text{and}\,\,\, x_g=\sum_{j\in J_g(v)}\frac{t_j}{\lambda_g} v_j$$ if $J_g\neq \emptyset$, while one sets $\lambda_g(v)=0$ otherwise. In particular, observe that $\iota(v_j^{g_j})=v_j$,  where $v_j$ occurs in the $g_j$-th coordinate of the deleted join. As $g=g_j h^{-1}$ if and only if $g_j=gh$, it is easily verified that $\iota$ preserves the respective $G$-actions.

\subsection{Test space} 
For a finite group $G$, let $\mathbb{R}[G]=\{\sum r_g g \mid r_g\in \mathbb{R}\}$ denote the \textit{right} regular representation, that is, with \textit{left} $G$-action arising from multiplication on the right by inverses. We let $\mathbb{R}^d[G]$ denote the $d$-fold sum of $\mathbb{R}[G]$ equipped with the diagonal action. For convenience, we denote each $(\sum_{g\in G} r_g^1 g,\ldots, \sum_{g\in G} r_g^d g)$ in $\mathbb{R}^d[G]$ by $\sum_{g\in G} (r_1^g,\ldots, r_d^g)\mathbf{e}_g$, so that  $$\mathbb{R}^d[G]=\{\sum_g v_g\mathbf{e}_g\mid v_g\in \mathbb{R}^d\}.$$ As before, $h\cdot \mathbf{e}_g=\mathbf{e}_{gh^{-1}}$ for each $h,g\in G$, so that each $h\in G$ acts as a right permutation of elements of $\mathbb{R}^d[G]$, that is, $h\cdot \sum_g v_g\mathbf{e}_g=\sum_g v_{gh}\mathbf{e}_g$.

Now let $\rho\colon G\rightarrow O(d)$ be any orthogonal representation. For notational purposes, we denote the representation space by $V$ and let $\cdot$ denote the resulting left action. In what follows, we will need to establish the dimension of a certain subrepresentation $W$ of $\mathbb{R}[G]$. To that end, define $V_\rho=\{\sum_g (g\cdot v) \mathbf{e}_g\mid v\in \mathbb{R}^d \}$ and let $$W_\rho=\left\{\sum_g w_g\mathbf{e}_g\mid \sum_g g^{-1}\cdot w_g = 0\right\}.$$ 

\begin{proposition}
\label{prop:1}
$V_\rho$ and $W_\rho$ are complementary subrepresentations of $\mathbb{R}^d[G]$. Moreover, $V_\rho\cong V$ and so $\dim W_\rho = d(|G|-1)$. 
\end{proposition} 

\begin{proof} 
Clearly, both $V_\rho$ and $W_\rho$ are linear subspaces of $\mathbb{R}^d[G]$. To see that they are subrepresentations, first consider $V_\rho$. For any $h\in G$ have $h\cdot \sum_g (g\cdot v) \mathbf{e}_g=\sum_g (g \cdot v) \mathbf{e}_{gh^{-1}}=\sum_g (gh\cdot v)\mathbf{e}_g=\sum_g g\cdot (h\cdot v)\mathbf{e}_g$. Thus $V_\rho$ is a subrepresentation of $\mathbb{R}^d[G]$ and in fact is isomorphic to $V$. As for $W_\rho$, we have $h\cdot \sum_g w_g\mathbf{e}_g= \sum_g w_{gh}\mathbf{e}_g$. Then $\sum_g g^{-1}\cdot w_{gh}= h\cdot \sum_g (gh)^{-1}\cdot w_{gh}=h\cdot \sum_g g^{-1}\cdot w_g =0$, so again $W_\rho$ is a subrepresentation of $\mathbb{R}^d[G]$.

	Now we show that $V_\rho$ and $W_\rho$ are complementary. First, note that $\sum_g (g \cdot v) \mathbf{e}_g$ lies in $W_\rho$ if and only if  $|G|v=0$, so $V_\rho\cap W_\rho=\{0\}$. On the other hand, suppose that $\sum_g v_g \mathbf{e}_g\in \mathbb{R}^d[G]$. Letting  $a=\frac{1}{|G|}\sum_g g^{-1}\cdot v_g$, we have that $b:=\sum_g (g\cdot a)\mathbf{e}_g$ lies in $V_\rho$. On the other hand, $\sum_g v_g \mathbf{e}_g-b$ lies in $W_\rho$ because $\sum_g g^{-1}\cdot(v_g -g\cdot a)=\sum_g (g^{-1}\cdot v_g -a)=0.$
\end{proof} 

Let $V_0=\{\sum_{g\in G} v \mathbf{e}_g\mid v\in \mathbb{R}^d\}$ denote the trivial representation inside $\mathbb{R}^d[G]$, and let  $$W_0=\left\{\sum_{g\in G} w_g\mathbf{e}_g\mid \sum_g w_g=0\right\}$$ denote its complementary subrepresentation. 
\begin{proposition}
\label{prop:W} 
Suppose that $V$ does not contain a trivial subrepresentation. Then\\
(i) $V_0\cap V_\rho=\{0\}$, and\\
(ii) $V':=V_0+V_\rho$ and $W:=W_0\cap W_\rho$ are complementary subrepresentations of $\mathbb{R}^d[G]$. In particular, $\dim W =d(|G|-2)$.
\end{proposition} 

\begin{proof}
To see (i), suppose that there exist $v_0$ and $v_1$ in $\mathbb{R}^d$ such that $g\cdot v_0=v_1$ for all $g\in G$. Certainly $v_1=v_0$, and so $v_0$ is  $G$-invariant. Thus $v_0=0$, since $V$ does not contain a trivial subrepresentation. 

For (ii), we first show that the intersection of $V'$ and $W$ is trivial. Suppose that $v_0,v_1\in V$ and let $w:=\sum_g (v_0+g \cdot v_1)\mathbf{e}_g\in V'$. We have $\sum_g g^{-1}\cdot (v_0+g\cdot v_1)=\sum_g g^{-1}\cdot v_0 + |G|v_1=|G|v_1$ because $\sum_g g^{-1} \cdot v_0$ is $G$-invariant. Thus $v_1=0$ if $w$ lies in $W_\rho$. Similar reasoning shows that $\sum_g (v_0+g\cdot v_1)=|G|v_0$, so that $v_0=0$ if $w\in W_0$. To conclude that $\mathbb{R}^d[G]=V'\oplus W$, we show that $\dim(V')+\dim W\geq d|G|$. But $\dim V'=2d$, while $\dim W= \dim W_0+\dim W_\rho -\dim(W_0+W_\rho)\geq d(|G|-1)+d(|G|-1)-d|G|=d(|G|-2)$ by Proposition~\ref{prop:1}.
\end{proof}

\subsection{Test map} 

As motivation for our test map, we first make a simple observation. Let $\rho\colon G\rightarrow O(d)$ be an orthogonal representation. Fixing an ordering of $G$, let $(\Delta_{N-1})^{\times G}_\Delta$  denote the $|G|$-fold \textit{deleted product} of $\Delta_{N-1}$ consisting of all $|G|$-tuples $x=(x_g)_{g\in G}$ of points from $\Delta_{N-1}$ with pairwise disjoint support. For any $x\in (\Delta_{N-1})^{\times G}_\Delta$, we let $$c_\rho(x)=\frac{1}{|G|}\sum_{g\in G}g^{-1}\cdot f(x_g).$$  In particular, $c_\rho$ is simply the average $$c_0(x)=\frac{1}{|G|}\sum_{g\in G}f(x_g)$$ when $\rho$ is the trivial representation. 
 
 \begin{proposition}
\label{prop:Fourier} Suppose that $\rho\colon G\rightarrow O(d)$ does not contain a trivial subrepresentation and let $f\colon\Delta_{N-1}\rightarrow \mathbb{R}^d$ be an affine map. If there exists some $(x_g)_{g\in G}\subset\Delta_{N-1}$ with pairwise disjoint support which satisfies equation~\ref{eqn:orbit}, then $a=c_0(x)$ and $u=c_\rho(x)$. 
\end{proposition}

\begin{proof} Assuming equation~\ref{eqn:orbit} holds, we have $|G|c_0(x)=a|G|+\sum_{g\in G} g\cdot u$.  The sum $\sum_{g\in G} g\cdot u=0$ is $G$-invariant and $\rho$ does not contain a trivial subrepresentation, so $a=c_0(x)$. Likewise, we have $\sum_g g^{-1}\cdot f(x_g)=\sum_{g\in G}g^{-1}\cdot a+|G|u$. Again $\sum_{g\in G}g^{-1}\cdot a$ is $G$-invariant, so $u=c_\rho(x)$. \end{proof} 

Let $f\colon \Delta_{N-1}\rightarrow \mathbb{R}^d$ be an affine map and suppose that $\rho\colon G\rightarrow O(d)$ does not contain a trivial subrepresentation. For each $\lambda x=\sum_{g\in G} \lambda_g x_g \in (\Delta_{N-1})^{\ast G}_\Delta$, we define
\begin{equation}
\label{eqn:join} 
c_\rho(\lambda x)=\frac{1}{|G|}\sum_{g\in G} \lambda_g g^{-1}\cdot f(x_g)\end{equation} and 
$$c_0(\lambda x)=\frac{1}{|G|}\sum_{g\in G} \lambda_gf(x_g).$$

Observe that the deleted product $(\Delta_{N-1})^{\times G}_\Delta$ can be identified as the subcomplex of the deleted join $(\Delta_{N-1})^{\ast G}_\Delta$ consisting of all $\lambda x$ with $\lambda_g=\frac{1}{|G|}$ for all $g\in G$. In that case, we have $c_\rho(x)=|G|c_\rho(\lambda x)$ and $c_0(x)=|G|c_0(\lambda x)$  as well. 
	
	Define $F\colon (\Delta_{N-1})_{\Delta}^{\ast G}  \rightarrow \mathbb{R}^d[G]$  by $F(\lambda x)= \sum_g F_g(\lambda x) \mathbf{e}_g$, where  $$F_g(\lambda x)= \lambda_gf(x_g)-c_0(\lambda x) - g\cdot c_\rho(\lambda x).$$ It follows that $F(\lambda x\cdot h)=\sum_g F_{gh}(\lambda x) \mathbf{e}_g=\sum_g F_g(\lambda x)\mathbf{e}_{gh^{-1}}=h\cdot F(\lambda x)$ for all $h\in G$, and therefore $F$ is equivariant with respect to the given group actions.

It is straightforward to show that image of $F$ actually lies in the representation $W=W_0\cap W_\rho$ defined in Proposition~\ref{prop:W}. Namely, $\sum_{g\in G}g\cdot c_\rho(\lambda x)=0$ because the sum is $G$-invariant, and hence $\sum_g F_g(\lambda x)=\sum_g \lambda_g f_g(x) -\sum_g c_0(\lambda x)-\sum_g g \cdot c_\rho (\lambda x)=|G|c_0(\lambda x)-|G|c_0(\lambda x)-0=0.$ Thus $F(\lambda x)\in W_0$.  Likewise, $\sum_g g^{-1} \cdot F_g(\lambda x)=\sum_g \lambda_g g^{-1}\cdot f_g(x) -\sum_g g^{-1}\cdot c_0(\lambda x)-\sum_g c_\rho (\lambda x)=|G|c_\rho(\lambda x)-0-|G|c_\rho(\lambda x)=0$ because $\sum_g g^{-1}\cdot c_0(\lambda x)$ is $G$-invariant. Thus $F(\lambda x)\in W_\rho$. 
							
Finally, let $\mathbb{R}^\perp[G]=\{\sum_{g\in G} r_g g\mid \sum_{g\in G} r_g=0\}$ denote the orthogonal complement of the right regular representation $\mathbb{R}[G]$. To preclude empty faces, we define $R\colon (\Delta_{N-1})^{\ast G}_{\Delta}\rightarrow \mathbb{R}[G]^\perp$ by $$R(\lambda x) = \sum_{g\in G} \left(\lambda_g-\frac{1}{|G|}\right)\, g.$$ This map is also $G$-equivariant, so that $$F\oplus R\colon (\Delta_{N-1})^{\ast G}_\Delta \rightarrow W\oplus \mathbb{R}^\perp[G]$$ is as well. 

Crucially, any collection $(x_g)_{g\in G}\subset \Delta_{N-1}$ with pairwise disjoint support that satisfies equation~\ref{eqn:orbit} is equivalent to a zero of $F\oplus R$. Indeed, the vanishing of $F\oplus R$ means that $F_g(\lambda x)=0$ and $\lambda_g=\frac{1}{|G|}$ for all $g\in G$, or equivalently that there exists some $x=(x_g)_{g\in G}$ in the deleted product $(\Delta_{N-1})^{\times G}_\Delta$ such that $f(x_g)=c_0(x)+g\cdot c_\rho(x)$ for all $g \in G$. Conversely, any $x\in (\Delta_{N-1})^{\times G}_\Delta$ satisfying equation~\ref{eqn:orbit} corresponds to a zero of $F\oplus R$ by Proposition~\ref{prop:Fourier}. 

To show that $F\circ R$ must have a zero, let $\iota\colon G^{\ast N}\rightarrow (\Delta_{N-1}^{\ast G})_\Delta$ be the identification isomorphism above. A straightforward argument below will show that the composition \begin{equation}
\label{eqn:L} 
L:=(F\oplus R)\circ \iota\colon G^{\ast N}\rightarrow W\oplus \mathbb{R}^\perp[G] \end{equation} is affine, that is, that $L(\sum_{j=1}^Nt_jv_j^{g_j})=\sum_{j=1}^N t_jL(v_j^{g_j})$ for any  $v=\sum_{j=1}^Nt_jv_j^{g_j}$ in $G^{\ast N}$. Clearly, $L$ is $G$-equivariant with respect to the actions considered. The representation $W\oplus \mathbb{R}^\perp[G]$ does not contain a trivial subrepresentation, and has dimension $t(|G|,d)-d-1$ by Proposition~\ref{prop:W}. Letting $N=t(|G|,d)-d$, the map $L$ is therefore  guaranteed to have a zero by the following ``Linear Borsuk--Ulam'' Theorem ~\cite[Theorem 2.4]{S2000} of Sarkaria. This proves the existence of an orbit  partition as claimed Theorem~\ref{thm:affine}.

\begin{theorem}\label{thm:Sarkaria} Let $N\geq 1$ be an integer and let $U$ be a real $(N-1)$-dimensional linear representation of $G$ which does not contain a trivial subrepresentation. If $L\colon G^{\ast N} \rightarrow U$ is an affine $G$-equivariant map, then there exists some $v\in G^{\ast N}$ such that $L(v)=0$.\end{theorem}

To prove that the mapping $L$ of equation~\ref{eqn:L} is affine, we show that this is the case for both $F\circ \iota$ and $R\circ \iota$. The argument for the later was given in the proof of Theorem 3.1 from ~\cite{LS21}. For the former, it suffices to prove that $F_1\circ \iota$ and $F_2\circ \iota$ are affine, where $F_1(\lambda x)=\sum_g \lambda_g f(x_g)\mathbf{e}_g$ and  $F_2(\lambda x)=\sum_g \lambda_g g^{-1}\cdot f(x_g)\mathbf{e}_g$. We show the proof for $F_1$; the proof for $F_2$ is nearly identical.  Consider $v=\sum_{j=1}^N t_j v_j^{g_j}\in G^{\ast N}$. We have $(F_1\circ \iota )(v_j^{g_j}))=f(v_j)\mathbf{e}_{g_j}$, so $(F_1\circ \iota)(v)=\sum_g \lambda_g f(x_g)\mathbf{e}_g=\sum_{g\in G}\lambda_g[\sum_{j\in J_g(v)} \frac{t_j}{\lambda_g}f(v_j)]\mathbf{e}_g=\sum_{j=1}^N t_jf(v_j)\mathbf{e}_{g_j}=\sum_{j=1}^N t_j(F_1\circ \iota)(v_j^{g_j})$ because $f$ is affine.

\subsection{Optimality} To complete the proof of Theorem~\ref{thm:affine}, we show that an orbit partition fails for generic affine maps $f\colon\Delta_{N-1}\rightarrow \mathbb{R}^d$ if $N<t(|G|,d)-d$. This follows from a standard (co-)dimension counting argument.

Let $N=t(r,d)-d-1$. Given any affine map $f\colon\Delta_{N-1}\rightarrow \mathbb{R}^d$, we have seen that the  existence of some collection $(x_g)_{g\in G}$ from $\Delta_{N-1}$ with pairwise support which satisfies equation~\ref{eqn:orbit} is equivalent to a zero of the map $L\colon G^{\ast N}\rightarrow W\oplus \mathbb{R}^\perp[G]$ of equation~\ref{eqn:L}. Let $\Sigma_{N-1}=\{s=(s_1,\ldots, s_N)\in \mathbb{R}^N\mid \sum_{j=1}^N s_j=1\,\, \text{and}\,\, s_j\geq 0\,\,\text{for all}\,\, j\}$ denote the standard $(N-1)$-simplex in $\mathbb{R}^N$. If the images $f(v_j)$ of the vertices of $\Delta_{N-1}$ are generic, we show that any $v=\sum_{j=1}^N t_j v_j^{g_j}$ satisfying $L(v)=0$ forces $t=(t_1,\ldots, t_N)$ in $\Sigma_{N-1}$ to satisfy $N$ linearly independent conditions, a contradiction. 

Consider $(g_1,\ldots, g_N)\in G^{\times N}$ for $v=\sum_{j=1}^N t_j v_j^{g_j}$ above. It follows that $J_g:=\{j\mid g_j=j\}\neq \emptyset$ for each $g\in G$. That $(R\circ i)(v)=0$ means that $\sum_{j\in J_g}t_j=\frac{1}{|G|}$ for all $g\in G$, thereby giving $(|G|-1)$ linearly independent conditions on $t$. Secondly, $(F\circ \iota)(v)=0$ means that $t$ is zero of each of the linear maps $$M_g(s)=\sum_{j\in J_g} s_jf(v_j)-\sum_{j=1}^N \frac{s_j}{|G|}f(v_j)-g\cdot \sum_{j=1}^N \frac{s_j}{|G|}g_j^{-1}\cdot f(v_j)$$ defined on $\Sigma_{N-1}$. As $|G|\geq 3$, $N-1=(|G|-2)(d+1)> d$, and so for any particular $g\in G$ one has $d$ linearly independent conditions provided the $f(v_j)$ are chosen generically. On the other hand, the image of $M(s)=\sum_{g\in G} M_g(s)\mathbf{e}_g$ lies in the subspace $W$ of $\mathbb{R}^d[G]$ of Proposition~\ref{prop:W}. For generic $f(v_j)$, it follows from Proposition~\ref{prop:W} that $M(t)=0$ determines $d(|G|-2)$ linearly independent conditions. Moreover, these conditions are independent from those given by $(R\circ \iota )(v)=0$, again provided the $f(v_j)$ are generic. One therefore has $N$ independent conditions on $t$ in total. 

\section{Genericity Conditions} 
\label{sec:3}

We now show that the translated orbits $a+G\cdot u$ guaranteed by Theorem~\ref{thm:affine} satisfy two generic conditions (see, e.g.,  ~\cite[Lemma 2.11]{FL18} and~\cite[Lemma 4.2]{FL16}), namely that $u$ has minimal stabilizer and that $\Aff(G\cdot u)$ has maximum  dimension. 

\begin{lemma} 
\label{lem:generic}
Let $G$ be a finite group of order $r\geq 3$ and let $\rho\colon G\rightarrow O(d)$ be an orthogonal representation which does not contain a trivial subrepresentation. Let $K=\ker\rho$ and let $m=\max\{ \dim\Aff(G\cdot u)\mid u\in \mathbb{R}^d\}$. If $N=t(r,d)-d$, then the following is true for generic affine maps $f\colon\Delta_{N-1}\rightarrow \mathbb{R}^d$: If $(x_g)_{g\in G}$ are points from $\Delta_{N-1}$ with pairwise disjoint support that satisfy equation~\ref{eqn:orbit}, then (i) $\Stab(u)=K$ and (ii) $\dim\Aff(G\cdot u)=m$.
\end{lemma}

As with the proof of optimality in Theorem~\ref{thm:affine}, Lemma~\ref{lem:generic} follows from a degrees of freedom argument. Note that (i) of Lemma~\ref{lem:generic} means that $|G\cdot u|=|G|$ if the representation is faithful, so that each $P(G;u,a)=\Conv(a+G\cdot u)$ is a free orbit polytope. On the other hand, (ii) $m=d$ if the representation is full-dimensional. Thus Theorem~\ref{thm:affine} together with Lemma~\ref{lem:generic}  completes the proof of Theorem~\ref{thm:orbit}. 

\begin{proof}[Proof of Lemma~\ref{lem:generic}]
Let $f\colon\Delta_{N-1}\rightarrow \mathbb{R}^d$ be affine and let $L\colon G^{\ast N}\rightarrow W\oplus\mathbb{R}^\perp[G]$ be the associated map given by equation~\ref{eqn:L}. By Proposition~\ref{prop:Fourier}, equation~\ref{eqn:join}, and the fact that  $f$ is affine, it follows that any $u$ satisfying equation~\ref{eqn:orbit} is of the form $u=|G|c_\rho(\iota (v))=\sum_{j=1}^N t_jg_j^{-1}\cdot f(v_j)$, where $v=\sum_{j=1}^N t_j v_j^{g_j}$ satisfies $L(v)=0$. If $f$ is generic, we have seen in the proof of Lemma~\ref{lem:generic} that corresponding to this $(g_1,\ldots, g_N)\in G^{\times N}$ one has $(N-1)$ linearly independent conditions on $t=(t_1,\ldots, t_N)$ in $\Sigma_{N-1}$. 

For each $s\in \Sigma_{N-1}$, let $u(s)=\sum_{j=1}^N s_jg_j^{-1}\cdot f(v_j)$. For any $h\in G\setminus K$, consider the linear map $L_h(s)= u(s)-h\cdot u(s)$ defined on $\Sigma_{N-1}$. If the $g_j^{-1}\cdot f(v_j)$ are chosen to be in general position, then $L_h$ is not the zero map because $h$ is not in $K$. Moreover, so long as the $f(v_j)$ are generic, the condition that $L_h(t)=0$ is independent from those determined by $L(v)=0$. Thus $u$ is not fixed by any $h\in G\setminus K$, as otherwise $t$ would satisfy $N$ independent conditions.

For maximum dimensionality of $\Aff(G\cdot u)$, let $m$ be as in the statement of Lemma~\ref{lem:generic}. As in ~\cite[Lemma 4.2]{FL16}, let $|G|=r$, fix an ordering $g_1,\ldots, g_r$ of $G$, and consider the $(d+1)\times r$ matrix $$A(s)=\begin{pmatrix} g_1\cdot u(s)  & \cdots & g_r\cdot u(s)\\ 1 & \cdots & 1 \end{pmatrix}$$ defined on $\Sigma_{N-1}$. Each $m\times m$ sub-determinant of $A(s)$ determines a polynomial $P(s)$.  For generic $f(v_j)$, each of these polynomials is non-zero and the condition that $P(t)=0$ is  independent from the linear conditions determined by $L(v)=0$. Thus $\dim\Aff(G\cdot u)=m$, as otherwise $P(t)=0$ for each of these polynomials and so $t$ would satisfy at least $N$ independent conditions.\end{proof} 

\subsection{Generic Isometries} For the proof of Theorem ~\ref{thm:closed} and the calculation of some of the isometry groups in  Theorem~\ref{thm:3poly}, we will need to show that the  orbit polytopes of Theorem~\ref{thm:orbit}  satisfy a minimal symmetry condition. To state this formally, we follow ~\cite{FL16, FL18}, which introduced the notion of a generic linear symmetry group of a linear representation. An analogous group in the orthogonal setting was given in~\cite[Chapter 7]{Fr18}. 

Suppose that $\rho\colon G\rightarrow O(d)$ is a full-dimensional orthogonal representation.  By definition, the ``generating set'' $$\Gens(G):=\{u\in \mathbb{R}^d\mid \Aff(G\cdot u)=\mathbb{R}^d\}$$ is non-empty. Reasoning as in the proof of (ii) of Lemma~\ref{lem:generic}, it can be shown that those $u$ which do not lie in $\Gens(G)$ form a proper algebraic subset of $\mathbb{R}^d$. In particular, $\Gens(G)$ is a dense open set~\cite[Corollary 4.5] {FL16}. 

Let $S_G$ denote the set of all permutations of $G$. For any $u\in\mathbb{R}^d$, let $$\OSym(G, u)=\{\pi\in S_G\mid \text{there exists some} \, A\in O(d)\,\, \text{such that}\,\, A(g\cdot u)=\pi(g)\cdot u\,\,\text{for all}\,\, g\in G\}.$$ If $u\in \Gens(G)$, it is easily seen that any $A\in O(d)$ above is uniquely determined by the permutation $\pi$, and therefore $\OSym(G,u)$ can be identified with the isometry group $\Iso(P(G;u,a))$ of the orbit polytope. Note that for each $g\in G$, the permutation $\pi_g$ given by left multiplication by $g$ corresponds to $\rho_g$ under this identification. The \textit{generic isometry group} is defined by $$\OSym(G)=\bigcap_{u\in \Gens(G)}\OSym(G,u).$$ For $u\in\Gens(G)$, to say that $\OSym(G,u)=\OSym(G)$ therefore means that the $d$-dimensional orbit polytope $P(G;u,a)$ has minimal isometry group. 

	Considering the standard inner product on $\mathbb{R}^d$, an easy observation communicated to the authors by Frieder Ladisch shows that $\pi\in \OSym(G,u)$ if and only if $u$ is a zero of the polynomial $$P^\pi_{g,h}(u)=\langle g\cdot u, h\cdot u\rangle-\langle \pi(g)\cdot u, \pi(h)\cdot u\rangle$$ for all distinct $g,h\in G$. If $\pi\in \OSym(G)$, then each $P^\pi_{g,h}$ vanishes on $\Gens(G)$, a non-empty open set in $\mathbb{R}^d$. Therefore $P^\pi_{g,h}(u)$ is the zero polynomial and $\OSym(G)\leq \OSym(G,u)$ for any $u\in \mathbb{R}^d$. If $\pi\in S_G \setminus \OSym(G)$, on the other hand, then by definition there is some $u\in \Gens(G)$ with $\pi\notin \OSym(G,u)$ and so at least one of the $P^\pi_{g,h}(u)$ is non-zero. In particular, the set of all $u\in \mathbb{R}^d$ with $\OSym(G,u)>\OSym(G)$ is a proper algebraic subset of $\mathbb{R}^d$.

\begin{lemma}
\label{lem:iso} Let $G$ be a finite group of order $r\geq 3$, let $\rho\colon G\rightarrow O(d)$ be a faithful and full-dimensional representation, and let $N=t(r,d)-d$. Then the following holds for generic affine maps $f\colon\Delta_{N-1}\rightarrow \mathbb{R}^d\colon$   If $(x_g)_{g\in G}$ are points from $\Delta_{N-1}$ with pairwise disjoint support that satisfy equation~\ref{eqn:orbit}, then $\OSym(G,u)=\OSym(G)$. 
\end{lemma}

\begin{proof}[Proof of Lemma~\ref{lem:iso}] 
For an affine map $f\colon\Delta_{N-1}\rightarrow \mathbb{R}^d$, again let $L\colon G^{\ast N}\rightarrow W\oplus\mathbb{R}^\perp[G]$ be the linear map given by equation~\ref{eqn:L} and consider $u=\sum_{j=1}^N t_jg_j^{-1}\cdot f(v_j)$ where $v=\sum_{j=1}^N t_j v_j^{g_j}$ satisfies $L(v)=0$. As before, let $u(s)=\sum_{j=1}^N s_j g_j^{-1}\cdot f(v_j)$ for $s\in \Sigma_{N-1}$. For each $\pi\in S_G$ and any distinct $g,h\in G$, consider the polynomial $P_{g,h}(s)=P^\pi_{g,h}(u(s))$. If the $g_j^{-1}\cdot f(v_j)$ are chosen to be in general position, then $P_{g,h}(s)$ is the zero polynomial exactly when $P^\pi_{g,h}(u)$ is.  By the observations above, there is some $P^\pi_{g,h}$ which is not the zero polynomial precisely when $\pi\in S_G\setminus\OSym(G)$. For such $P^\pi_{g,h}$  and generic $f(v_j)$, the condition determined by $P_{g,h}(t)=0$ is independent from the  linear conditions determined by $L(v)=0$. Thus $\pi\notin \OSym(G,u)$ for any $\pi\notin \OSym(G)$, as otherwise $t\in \Sigma_{N-1}$ would satisfy $N$ independent conditions.\end{proof}

To prove Theorem~\ref{thm:closed}, we now only need to appeal to a result of ~\cite{FL16}. As there, we shall say that a faithful and full-dimensional representation $\rho\colon G\rightarrow O(d)$ is (orthogonally) \textit{closed} if $\OSym(G)$ consists only of those permutations determined by left multiplication in $G$. Thus a closed representation means precisely that $\Iso(P(G;u,a))=G$ whenever $u\in \Gens(G)$.  

By replacing orthogonal representations with arbitrary linear ones and the orthogonal group $O(d)$ with the general linear group $GL(\mathbb{R}^d)$, one has the groups $\LinSym(G,u)$ and $\LinSym(G)$ of~\cite{FL16} analogous to $\OSym(G,u)$ and $\OSym(G)$ above. One likewise has an analogous definition of linear closure. Theorem A of ~\cite{FL16} shows that the inclusion representation $G\hookrightarrow GL(\mathbb{R}^d)$ is linearly closed whenever $G\le GL(\mathbb{R}^d)$ is the affine symmetry group of a $d$-dimensional vertex transitive polytope. Of course, linear closure of an orthogonal representation implies orthogonal closure of the representation. Thus if $G<O(d)$ is the isometry group of a vertex-transitive $d$-polytope, the inclusion representation is automatically (orthogonally) closed. The proof of Theorem~\ref{thm:closed} follows immediately. 

\begin{proof}[Proof of Theorem~\ref{thm:closed}] Let $P(r,d)=\Conv(a+G\cdot u)$ be the $d$-dimensional free orbit polytope guaranteed by Theorem~\ref{thm:orbit}. By Lemma~\ref{lem:iso}, we have $\OSym(G,u)=\OSym(G)$. As the inclusion representation $G\hookrightarrow O(d)$ is closed, $\OSym(G)$ is given by left multiplication and therefore $\Iso(P(G;u,a))=G$. \end{proof} 

Another class of linearly closed representations $\rho\colon G\rightarrow GL(\mathbb{R}^d)$ are those which are \textit{absolutely irreducible}~\cite[Theorem 5.5] {FL16}, that is, irreducible when viewed as a representation over the complex numbers. In particular, any absolutely irreducible orthogonal representation $\rho\colon G\rightarrow O(d)$ must be (orthogonally) closed. Any real irreducible representation is  full-dimensional, and in fact $\Gens(G)=\mathbb{R}^d\setminus\{0\}$ because the linear subspace spanned by any orbit is a $G$-invariant. Note also that $|G|\geq 3$ for any (real) irreducible representation $\rho: G\rightarrow O(d)$ if $d\geq 2$.

\begin{corollary} 
\label{cor:absolute}
Let $d\geq 2$ and let $\rho\colon G\rightarrow O(d)$ be a faithful and absolutely irreducible representation of a finite group $G$ of order $r$. Then a generic set of $(r-2)(d+1)+2$ points in $\mathbb{R}^d$ can can be $P(r,d)$-partitioned by a vertex-transitive polytope with $\Iso(P(r,d))=G$.  
\end{corollary}

\section{Examples}
\label{sec:4}

\subsection{Finite Subgroups of Spheres}
\label{sec:4.1}

Let $\mathbb{F}$ be either the field $\mathbb{C}$ of complex numbers or the skew field $\mathbb{H}$ of quaternions. The corresponding unit spheres $S(\mathbb{F})$ from groups under $\mathbb{F}$-multiplication. Left multiplication by any  $u\in S(\mathbb{F})$ defines an $\mathbb{F}$-isometry, so that $S(\mathbb{C})=S^1=U(1)$ and $S(\mathbb{H})=S^3$ are the 1-dimensional unitary and symplectic groups, respectively.

For any finite subgroup $G$ of $S(\mathbb{F})$, left multiplication of $\mathbb{F}$ by $G$ defines a faithful representation, and in fact $\Stab(u)$ is trivial for all non-zero $u\in \mathbb{F}$. The corresponding free orbit polytopes are the translates of $\Conv(G)u$ for any $u\neq 0$, or equivalently the translated and scaled $\mathbb{F}$-unitary copies of $\Conv(G)$. We therefore expect the upper bound $N_{P(G)}\leq (|G|-2)(d+1)+2$ given by Theorem~\ref{thm:orbit} for $P(G)=\Conv(G)$ to exceed the exact value by $\dim(O(d))-\dim S(\mathbb{F})$. This will be confirmed below when $\mathbb{F}=\mathbb{C}$ (in which case the difference is zero) and is left as a conjecture when $\mathbb{F}=\mathbb{H}$ (in which case the difference is 3).

\subsubsection{Regular Polygons } The non-trivial finite subgroups $G<S^1$ are precisely the cyclic groups $\mathbb{Z}_r$ realized as the $r$-th roots of unity. Letting $\omega_r=e^{2\pi i/r}$ denote the standard $r$-th root of unity, $\mathbb{Z}_r$ is the vertex set  $\{\omega_r^k\mid 0\leq k<r\}$ of the regular $r$-gon $\Conv(\mathbb{Z}_r)$ inscribed in the unit circle when $r\geq 3$. In particular, the inclusion $ \mathbb{Z}_r\hookrightarrow U(1)\cong SO(2)$ is full-dimensional when $r\geq 3$ (and in fact irreducible). Any regular $r$-gon in the plane is of the form $P(\mathbb{Z}_r;u,a)$ for some $u\neq 0$ and $a\in \mathbb{C}$, so Theorem~\ref{thm:orbit} recovers Theorem~\ref{thm:polygon} when $r\geq 3$.

\subsubsection{Regular 4-polytopes} 
 Proceeding analogously, we consider the finite subgroups of $S^3$. It is a classical fact that these are the binary dihedral groups $D_r^*$ for all $r\geq 2$ and the binary polyhedral groups $T^*,O^*,$ and $I^*$, the pullbacks under the canonical 2-fold covering homomorphism $\phi\colon SU(2)\rightarrow SO(3)$ of the Dihedral groups $D_r$ (viewing $D_2$ as $\mathbb{Z}_2\times \mathbb{Z}_2$) and of the rotation groups $T,O$ and $I$ of the regular tetrahedron, octahedron, and icosahedron, respectively. Each inclusion $ G\hookrightarrow Sp(1)\cong SU(2)\subset SO(4)$ is irreducible and so is in particular full--dimensional.
	
	Viewing $\mathbb{Z}_r$ as the $r$-th roots of unity in $\mathbb{C}$ as before, one has the explicit decomposition $D_r^*=\mathbb{Z}_{2r}\cup \mathbb{Z}_{2r}j\subset \mathbb{H}$. Thus $\Conv(D_r^\ast)=P_{2r} \oplus P_{2r}j$ is the join of two congruent regular $2r$--gons lying in orthogonal planes  in $\mathbb{R}^4$ which intersect at the origin. In particular, $D_4$ is the Quaternion group $Q_8=\{\pm 1, \pm i, \pm j, \pm k\}$, and $\Conv(Q_8)$ is the standard regular 4-dimensional cross-polytope. This gives part (a) of Theorem~\ref{thm:4poly}. When $r\geq 3$, the $P(D_r^*)$ are vertex and facet transitive ``fusils"~\cite{J18} with $\Iso(P(D_r^*))\cong D_{2r}\times D_{2r}\times \mathbb{Z}_2$, whose tetrahedral facets are composed of congruent isosceles triangle faces. For parts (b) and (c) of Theorem~\ref{thm:4poly}, we note that $T^*$ consists of the 24 units of the Hurwitz Quaternion ring, and these are the vertices of the standard regular 24--cell. Likewise, the 120 vertices of $I^*$ are the vertex set of a regular 600--cell (see, e.g,~\cite[Chapter 4]{Du64}).

\subsection{Coxeter Permutahedra and Alternahedra}

 Recall that a $d$-dimensional finite reflection group is a finite subgroup of $O(d)$ which is generated by reflections across linear hyperplanes. A finite reflection group  is called \textit{irreducible} if the inclusion representation is irreducible. The irreducible finite reflection groups are precisely the isometry groups of certain uniform polytopes, including those of all regular polytopes, while all other finite reflection groups are the products of these (see, e.g.,~\cite[Chapter 2]{Hu92}). 
 
It is easily seen that an orbit polytope $P(G;u,a)$ of a finite reflection group is free precisely when $u$ is not contained in any linear hyperplane determined by a reflection of $G$. The free orbit polytopes are commonly called Coxeter permutahedra, with the classical permutahedra (see, e.g., ~\cite{Zi95}) arising from the standard representation of the symmetric group. All Coxeter permutahedra for a given reflection group are combinatorially equivalent (see, e.g., ~\cite[Proposition 2.1]{Ho12}, which need not be true for free orbit polytopes of a given representation in general (see, e.g.,~\cite{On93}). 

Owing to the classification of finite reflection groups, it follows quickly that $\Iso(P(G;u,a))=G$ for the the Coxeter permutahedra guaranteed by Theorem~\ref{thm:orbit}. For irreducible groups, this is an immediate consequence of Theorem~\ref{thm:closed}. For a direct product $G=\prod_{j=1}^k G_j$ of irreducible $d_j$-dimensional reflection groups $G_j$, one has  $P(G;u,a)=\prod_{j=1}^kP(G_j;u_j,a_j)$, for any $u=(u_1,\ldots, u_k)\in \oplus_{j=1}^k \mathbb{R}^{d_j}\setminus\{0\}$ and $a=(a_1,\ldots, a_k) \in \oplus_{j=1}^k \mathbb{R}^{d_j}$. If $u_j\neq0$ for all $1\leq j\leq k$, then each $P(G_j;u_j,a_j)$ is full dimensional and so is $P(G;u,a)$. On the other hand, $\Iso(P(G;u,a))=\prod_{j=1}^k \Iso(P(G;u_j,a_j))=\prod_{j=1}^k G_j=G$ for a generic choice of $u$.  Thus the result again follows from Theorem~\ref{thm:closed}.

\begin{corollary} Let $d\geq 2$, let $G<O(d)$ be a finite reflection group, and let $r=|G|$. Then a generic set of $(r-2)(d+1)+2$ points in $\mathbb{R}^d$ can be $P(r,d)$-partitioned by a Coxeter permutahedra with isometry group $G$. 
\label{cor:Coxeter} 
\end{corollary}

For irreducible $d$-dimensional reflection groups, we note that the isometry group of any Coxeter permutahedra of Corollary~\ref{cor:Coxeter} is precisely $G$ so long as $d\geq 3$ and $G$ is neither (1) the isometry group $S_d$ of the $(d-1)$-simplex nor (2) one of the exceptional groups $F_4$ or $E_6$~\cite[Chapter 11.6]{J18}. 
 
Given a finite reflection group $G$, we also consider its index two rotational subgroup $G^+$, whose corresponding free orbit polytopes we shall call Coxeter alternahedra (see ~\cite{CrKe06} for when $G=S_n$ is the symmetric group). Provided one takes $u\in \mathbb{R}^d\setminus\{0\}$ to be trivially stabilized by the full reflection group $G$, the resulting alternahedron is obtained by ``alternation"~\cite[Chapter 8.6]{Cox73} of the Coxeter permutahedron, that is, by taking the convex hull of those vertices lying on the permutahedron which are indexed by the subgroup $G^+$. An easy adaption of the the proof of Lemma~\ref{lem:generic} shows that one can ensure that $u$ has trivial stabilizer with respect to all the full group $G$ and not just $G^+$, so that the alternahedra guaranteed by Theorem~\ref{thm:orbit} are always alternated permutahedra.

The isometry groups in (f)--(h) of Theorem~\ref{thm:3poly} are given by the following proposition:

\begin{proposition} 
\label{prop:alternahedra} 
Let $d\geq 3$, let $G^+<SO(d)$ be the rotational subgroup of an irreducible finite reflection group $G<O(d)$, and let $r=|G^+|$. Then a generic set of $(r-2)(d+1)+2$ points in $\mathbb{R}^d$ can be $P(r,d)$-partitioned by a Coxeter alternahedra with isometry group $G^+$.
\end{proposition} 

 We note that Proposition~\ref{prop:alternahedra} is false when $d=2$, as follows by considering the Dihedral group $D_r$ when $r\geq 3$, in which case $D_r^+=\mathbb{Z}_r$ and the corresponding alternahedra are regular $r$-gons.

\begin{proof}[Proof of Proposition~\ref{prop:alternahedra}]

Let $\rho\colon G\hookrightarrow O(d)$ be the inclusion and let $\rho^+\colon G^+\hookrightarrow SO(d)$ be its restriction to $G^+$. We show that $\rho^+$ is absolutely irreducible and appeal to Proposition~\ref{cor:absolute}. The proof is a basic character theory argument along the lines of~\cite[Proposition 5.1]{FH04}. 
 
 Viewing $\rho$ and $\rho^+$ as complex representations, denote their characters by $\chi$ and $\chi^+$, respectively. As shown in~\cite{DwWi01}, $\rho$ is in fact complex irreducible, so that $\langle \chi, \chi\rangle=|G|$. Thus $2|G|^+=\langle \chi, \chi\rangle =\sum_{g\in G}|\chi(g)|^2=\sum_{r\in G^+} |\chi^+(r)|^2+\sum_{s\notin G^+}|\chi(s)|^2$. Expressing $\chi^+=\sum_{j=1}^k d_i \chi_i^+$ as the sum of its irreducible characters, linearity and orthogonality of characters gives $\sum_{r\in G^+} |\chi^+(r)|^2= \langle\chi^+,\chi^+\rangle=(\sum_{j=1}^k d_j^2)|G^+|$. Thus we either have $\langle \chi^+,\chi^+\rangle=|G^+|$, in which case $\rho^+$ is complex irreducible, or else that $\langle \chi^+, \chi^+\rangle = 2|G^+|$, in which case $\chi(s)=0$ for each reflection $s\in G\setminus G^+$. As $d\geq 3$ and $\chi(s)=d-2$ for any reflection, we conclude that $\rho^+$ is complex irreducible.\end{proof}

\subsection{Polyhedral Examples} 
\label{sec:4.3}
We now consider the polyhedral partitions of Theorem~\ref{thm:3poly}. 

\subsubsection{Regular Prisms and Antiprisms}

For part (a), let $n\geq 3$ and let $\rho\colon \mathbb{Z}_n\oplus \mathbb{Z}_2\rightarrow O(3)$ be the product of the inclusions $\mathbb{Z}_n\hookrightarrow U(1)\cong SO(2)$ and $\mathbb{Z}_2=\{\pm1\}=O(1)$. For $u=(u_1,u_2)$ and $a=(a_1,a_2)$ in $\mathbb{C}\times \mathbb{R}$, the orbit polytope $P(\mathbb{Z}_n\oplus \mathbb{Z}_2;u,a)$ is free provided $u_1,u_2\neq 0$, in which case $P(\mathbb{Z}_n\oplus \mathbb{Z}_2;u,a)=P(\mathbb{Z}_n;u_1,a_1)\times P(\mathbb{Z}_2;u_2,a_2)$ is a right (not necessarily uniform) regular prism $Pr(n)$ with regular $n$-gon base.  This is a special case of the ``multiprism" partitions of~\cite[Theorem 1.3]{LS21}. 

For part (b), let $n\geq 2$ and let $\rho\colon \mathbb{Z}_{2n}\rightarrow O(3)$ be the sum of inclusion representation  $\mathbb{Z}_{2n}\hookrightarrow U(1)$ and the representation $\chi\colon \mathbb{Z}_{2n}\rightarrow O(1)$ defined by $\chi(\omega_{2n}^k)=(-1)^k$ for all $0\leq k<2n$. As before, one has a free orbit polytope provided each $u_i$ of $u=(u_1,u_2)\in \mathbb{C}\times\mathbb{R}$ is non-zero. For $n\geq 3$, the translation of an orbit $\mathbb{Z}_{2n}\cdot u$ consists of the vertices of a regular $n$-gon and those of a $\pi/n$-rotated copy of this $n$-gon lying directly above it in a parallel plane. The convex hull is thus a right regular antiprism $APr(n)$, whose faces consist of the two congruent regular $n$-gons and $2n$ congruent isosceles triangles. In particular, $APr(3)$ is an octahedron. When $n=2$, the regular $n$-gons degenerate to perpendicular congruent edges, so that $APr(2)$ is a tetrahedron with congruent isosceles triangle faces.

    The regular $n$-gon bases for $Pr(n)$ and $APr(n)$ guaranteed by Theorem~\ref{thm:orbit} are precisely those which are parallel to the coordinate plane $\mathbb{R}^2\times 0$, in which case the upper bounds given in parts (a) and (b) of Theorem~\ref{thm:3poly} are tight. As there are $\mathbb{R}P^2$ choices of linear planes in $\mathbb{R}^3$, we therefore expect the upper bounds given there to exceed the true value by 2. Owing to the single extra condition on edge length, we conjecture that $N_{Pr(n)}=N_{Ar(n)}=8n-7$ for uniform prisms and antiprisms. 
    
\subsubsection{Omnitruncated and Snub Polyhedra}
The polyhedra of parts (c)--(h) of Theorem~\ref{thm:3poly} are the Coxeter permutahedra and their alternahedra arising from the irreducible 3-dimensional finite reflection groups. These are the isometry groups $\bf{T}, \bf{O}$, and $\bf{I}$ of the Platonic solids, with orders 24, 48, and 120, respectively. The resulting permutahedra are the (not necessarily uniformly) omnitruncated polyhedra~\cite{J66} -- the truncated icosahedron $tO$, truncated cuboctahedron $tCO$, and truncated icosidodecahedron $tID$, respectively (see, e.g., ~\cite[Example 2.7]{Ho12}). To describe their faces, we note that all $d$-dimensional Coxeter permutahedra are $d$-simple~\cite[Theorem 2.3]{Ho12} and that each $j$-face is itself a $j$-dimensional Coxeter permutahedra~\cite[Proposition 2.1]{Ho12}. The 2-dimensional finite reflection groups are precisely the Dihedral groups $D_r$, $r\geq 2$ (where again we view $D_2$ as $\mathbb{Z}_2\times \mathbb{Z}_2$), and it is easily seen that the Coxeter permutahedra for these groups are equiangular $2r$-gons with alternating congruent edges. The facets for parts (c)--(e) are therefore as follows: a single set of congruent rectangles and two sets of congruent hexagons for $tO$;  a set each of congruent rectangles, hexagons, and octagons for $tCO$; and a set each of congruent rectangles, hexagons, and decagons for $tID$.  Alternation of these permutahedra gives the (chiral) snub polyhedra: the icosahedron $sT$ realized as a snub tetrahedron, the snub cube $sC$, and the snub dodecahedron $sD$, respectively, whose faces can be described as follows: in addition to a single class of congruent triangles, one has two sets of congruent equilateral triangles for $sT$; a set of congruent equilateral triangles and a set of congruent squares for $sC$; and a set of congruent equilateral triangles and a set of congruent regular pentagons for $sI$. 

Considering conjugation of the above groups by orthogonal matrices, we expect the estimates of parts (c)--(h) of Theorem~\ref{thm:3poly} to exceed the exact value by $\dim(O(3))=3$ in each case. As there are three distinct edge lengths, two additional conditions are required to ensure uniform polyhedra, and so we again expect the exact values in the uniform setting to be one fewer than the upper bounds given by Theorem~\ref{thm:3poly}.

\subsection{Non-Faithful Representations}
\label{sec:4.4} As an example of Theorem~\ref{thm:affine} in the non-faithful setting, we provide a generalization of Theorem~\ref{thm:polygon} with intersecting convex hulls. For partitions with a somewhat similar spirit, see~\cite{S16}.  

\begin{theorem}
\label{thm:Tobias}
Let $r=r_1r_2$, where $r_1\geq 3$ and $r_2\geq 1$. Then a generic set of $3r-4$ points in $\mathbb{R}^2$ can be partitioned into a collection of $r$ subsets $\{A_i^j\}_{1\leq i \leq r_1, 1\leq j \leq r_2}$ such that\\
(a) $I_i:=\cap_{j=1}^{r_2}\Conv(A_i^j)\neq\emptyset$ for all $1\leq i \leq r_1$, and\\
(b) there exist points $x_1\in I_1,\ldots x_{r_1}\in I_{r_1}$ so that $\{x_1,\ldots, x_{r_1}\}$ is the vertex set of a regular $r_1$-gon.\\
Moreover, such a partition fails for a generic set of $N<3r-4$ points. 
\end{theorem} 

For instance, a generic set of 32 points in the plane can be partitioned into 12 subsets whose convex hulls intersect in 4 groups of 3, and there are four points, one from each intersection, which are the vertices a square. One may also divide the set into 12 subsets whose convex hulls intersect in 3 groups of 4 so that there are 3 points, one from each intersection, which are the vertices of an equilateral triangle.  

\begin{proof}[Proof of Theorem~\ref{thm:Tobias}]
 Consider the representation $\rho\colon \mathbb{Z}_r\rightarrow U(1)$  given by $\rho(\omega_r^k)=\omega_r^{kr_2}=\omega_{r_1}^k$. Stated in terms of affine mappings $f\colon\Delta_{N-1}\rightarrow \mathbb{C}$, the existence of a partition satisfying (a) and (b) of Theorem~\ref{thm:Tobias} is equivalent to the existence of a collection of points $\{x_k\}_{\in \mathbb{Z}_r}$ from the simplex with pairwise disjoint support and points $a$ and $u\neq0$ in the plane for which $f(x_k)=a+\omega_r^k\cdot u$ for all $k\in \mathbb{Z}_r$. Thus we are finished by Theorem~\ref{thm:affine} and Lemma~\ref{lem:generic}.  
 \end{proof}
 
 \subsection{A Conjecture for Vertex Transitive Polytopes} We conclude with a conjecture. While we have exclusively considered free orbit polytopes, any  vertex--transitive polytope is an orbit polytope given by its isometry group. 
  We expect the upper bound for free orbit polytopes arising from Theorem~\ref{thm:orbit} to carry over to the non-free setting.

\begin{conjecture}
\label{conj:transitive} $N_{P(r,d)}\leq t(r,d)-d$ for any vertex--transitive $d$-polytope $P(r,d)$ on $r$ vertices. 
\end{conjecture}

\section*{Acknowledgements}
The authors are grateful to the anonymous referee, whose many thoughtful suggestions improved the clarity and presentation of the manuscript. The authors also thank Frieder Ladisch for very useful discussions and Florian Frick for helpful comments.

\bibliography{bib}{}

\begin{thebibliography}{}


\bibitem{Ba77} L. Babai. Symmetry groups of vertex-transitive polytopes. \textit{Geom. Dedicata}, Vol. 6 (1977) 331--337. 

\bibitem{B82}  I. B\'ar\'any. A generalization of Carath\'eodory's theorem. \textit{Discr. Math.}, Vol. 40, No. 2--3 (1982) 141--152. 

\bibitem{BSo19} I. B\'ar\'any and P. Sober\'on. Tverberg's theorem is 50 years old: a survey. \textit{Bull. Amer. Math. Soc.}, Vol. 55, No. 4 (2018) 459--492. 

\bibitem{BaVe88} A. I. Barvinok and A. M. Vershik. Convex hulls of orbits of representations of finite groups, and combinatorial optimization. \textit{Func. Anal. Appl.}, Vol. 22, No. 3 (1988) 224--225.

\bibitem{BZ17} P. Blagojevi\'c and G.M. Ziegler. Beyond the Borsuk--Ulam theorem: The topological Tverberg story, \textit{A Journey through discrete mathematics.} Springer, Cham (2017) 273--341.

\bibitem{BoBo10} A. V. Borovik and A. Borovik. \textit{Mirrors and Reflections: The Geometry of Finite Reflection Groups.} Universitext. Springer, New York (2010).

\bibitem{Cox73} H. S. M. Coxeter. \textit{Regular Polytopes}. Dover Publications, United Kingdom (1973).

\bibitem{Cox34} H. S. M. Coxeter. Wythoff's construction for uniform polytopes. \textit{Proc. Lond. Math. Soc.}, Vol. 38 , No. 2 (1934) 327--339. 

\bibitem{CrKe06} J. Cruickshank and S. Kell. Rearrangement Inequalities and the Alternahedron. \textit{Discrete Comput. Geom.}, Vol. 35 (2006) 241--254. 

\bibitem{DELGMM19} J. A. De Loera, X. Goaoc, F. Meunier, and N. H. Mustafa. The discrete yet ubiquitous theorems of Carath\'eodory, Helly, Tucker, and Tverberg. \textit{Bull. Amer. Math. Soc.} Vol. 56 (2019), 415--511.

\bibitem{Du64} P. Du Val. \textit{Homographies, Quaternions, and Rotations}. Clarendon Press: Oxford University Press (1964). 

\bibitem{DwWi01} W. G. Dwyer and C. W. Wilkerson. Centers and Coxeter elements. \textit{Homotopy Methods in Algebraic Topology}. Contemp. Math. Vol. 271, American Mathematical Society (2001) 53--76.   

\bibitem{Fr18} E. Friese. \textit{Generic Symmetries of Group Representations}, Doctoral Dissertation. Universit\"at Rostock (2018). 

\bibitem{FL16} E. Friese and F. Ladisch. Affine Symmetries of orbit polytopes. \textit{Adv. Math}, Vol. 288 (2016) 386--425.

\bibitem{FL18} E. Friese and F. Ladisch. Classification of affine symmetries of orbit polytopes. \textit{J. Algebraic Combin.}, Vol. 48, No. 3 (2018) 481--509.

\bibitem{FH04} W. Fulton and J. Harris. \textit{Representation Theory: A First Course. Graduate Texts in Mathematics}. Springer, New York (2004). 

\bibitem{Ho12} C. Hohlweg. Permutahedra and associahedra: generalized associahedra from the geometry of finite reflection groups. \textit{Associahedra, Tamari lattices and related structures.} Prog. Math. Phys. 299. Birkh\"auser/Springer, Basel (2012) 129--159.

\bibitem{HoLaTh11} C. Hohlweg, C. E. M. C. Lange, and H. Thomas. Permutahedra and generalized associahedra. \textit{Adv. Math.}, Vol. 226, No.1 (2011) 608--640.

\bibitem{Hu92} J.E. Humphries. \textit{Reflection Groups and Coxeter Groups}. Cambridge University Press (1992). 

\bibitem{J18} N. W. Johnson. \textit{Geometries and Transformations.} Cambridge University Press (2018). 

\bibitem{J66} N.W. Johnson. \textit{The Theory of Uniform Polytopes and Honeycombs}. Doctoral Dissertation. University of Toronto (1966). 


\bibitem{LS21} L. Leiner and S. Simon. Regular polygonal partitions of a Tverberg type. \textit{Discrete Comp. Geom.}, Vol. 66 (2021) 1053--1071. 

\bibitem{MSV15} J. Morris, P. Spiga, and G. Verret. Automorphisms of Cayley graphs on generalised dicyclic groups. \textit{Eur. J. Comb.}, Vol. 43 (2015) 68--81. 

\bibitem{On93} S. Onn.  Geometry, complexity, and combinatorics of permutation polytopes. \textit{J. Combin. Theory Ser. A}, Vol. 64, No.1 (1993) 31--49.

\bibitem{PS14} M. Perles and M. Sigron. Strong general position. arXiv:1409.2899[math.CO] (2014). 

\bibitem{SSS11} R. Sanyal, F. Sottile, and B. Sturmfels. Orbitopes. \textit{Mathematika}, Vol. 57, No. 2 (2011) 275--314. 

\bibitem{S2000} K.S. Sarkaria. Tverberg partitions and Borsuk--Ulam theorems. \textit{Pacific J. Math.}, Vol. 196, No. 1 (2000) 231--241.

\bibitem{S16} S. Simon. Average-value Tverberg partitions via finite Fourier analysis. \textit{Isr. J. Math.}, Vol. 216, No. 2 (2016) 891--904. 

\bibitem{T66} H. Tverberg. A generalization of Radon's theorem. \textit{J. London Math. Soc.}, Vol. 41 (1966) 123--128.
 
 \bibitem{Zi95} G. M. Ziegler. \textit{Lectures on Polytopes}. Grad. Texts Math. 152, Springer-Verlag, New York (1995).
 
 




\end{thebibliography}
\bibliographystyle{plain}

\end{document}